\documentclass[11pt]{article}
 \usepackage[utf8]{inputenc}
 \usepackage[T1]{fontenc}
\usepackage[all]{xy}

\usepackage{tikz}
\usepackage{tikz-cd}

\usepackage{url}

\usepackage{amsmath,amstext,amsthm,amssymb,amscd,version}
\usepackage{mathtools}
\usepackage{xcolor}
\usepackage{comment}

\def\bA{\mathbb{A}}

\def\bZ{\mathbb{Z}}
\def\bQ{\mathbb{Q}}
\def\bR{\mathbb{R}}
\def\bC{\mathbb{C}}

\def\cO{\mathcal{O}}

\def\cD{\mathcal{D}}

\def\cL{\mathcal{L}}
\def\cM{\mathcal{M}}

\newtheorem{thmintro}{Theorem}

\newtheorem{thm}{Theorem}[section]
\newtheorem*{thm*}{Theorem}

\newtheorem{lem}[thm]{Lemma}
\newtheorem{prop}[thm]{Proposition}
\newtheorem{cor}[thm]{Corollary}
\newtheorem{rem}[thm]{Remark}

\newtheorem{defn}[thm]{Definition}
\newtheorem*{definition*}{Definition}

\newtheorem*{qtn}{Question}

\DeclareMathOperator{\HH}{\mathrm{H}}

\DeclareMathOperator{\SL}{SL}

\DeclareMathOperator{\GL}{GL}

\DeclareMathOperator{\Sh}{Sh}
\DeclareMathOperator{\sh}{sh}
\DeclareMathOperator{\Alb}{Alb}
\DeclareMathOperator{\alb}{alb}

\newcommand{\yo}[1]{{\color{purple} \sf [#1]}}

\title{Existence of the Shafarevich morphism for semisimple local systems on quasi-projective varieties}
\author{Yohan Brunebarbe}
\date{}

\begin{document}

\maketitle

\begin{abstract}
Let X be a normal connected complex algebraic variety equipped with a semisimple complex representation of its fundamental group. Then, under a maximality assumption, we prove that the covering space of X associated to the kernel of the representation has a proper surjective holomorphic map with connected fibres onto a normal analytic space with no positive-dimensional compact analytic subspace. 
\end{abstract}


\section{Introduction}

In an attempt to understand which complex analytic spaces can be realised as the universal covering of a complex algebraic variety, Shafarevich asked whether the universal covering $\tilde{X}$ of any smooth projective variety $X$ is holomorphically convex \cite[IX.4.3]{Shafarevich_book}. In other words, does there exists a proper holomorphic map $sh_{\tilde{X}} \colon \tilde{X} \to \Sh(\tilde{X})$ to a Stein analytic space $\Sh(\tilde{X})$? If such a map $sh_{\tilde{X}}$ exists, then one may assume in addition that it is surjective with connected fibres, and then it is unique: it is the so-called Cartan-Remmert reduction of $\tilde{X}$. Shafarevich question, nowadays known as Shafarevich conjecture, has been established when $\pi_1(X)$ is virtually nilpotent \cite{Katzarkov97} or when $\pi_1(X)$ has a faithful complex linear representation \cite{Eyssidieux-reductive, EKPR}. See also \cite{Kollar-Shafarevich, Campana94, Lasell-Ramachandran, Katzarkov-Ramachandran, CCE15}. The general case is however still open, due to the lack of methods to deal with non-linear fundamental groups. In contrary, when the fundamental group admits many linear representations, one can use the whole apparatus of tools coming from classical and non-abelian Hodge theory and the theory of harmonic maps towards buildings \cite{Corlette, Simpson_Higgs, Gromov-Schoen, Mochizuki_asterisque, Brotbek-Daskalopoulos-Deng-Mese}. \\

The goal of this paper is to study the following natural generalization of Shafarevich question.
\begin{qtn}\label{question Shafa}
Let $X$ be a (non-necessarily proper) normal complex algebraic variety and $H \subset \pi_1(X)$ a normal subgroup. Under which conditions the corresponding Galois étale covering $\tilde{X}^H$ of $X$ is holomorphically convex? 
\end{qtn}

When $H = \{ 1\}$, this question is answered positively in \cite{Green-Griffiths-Katzarkov} and \cite{Aguilar-Campana}, assuming either that $\pi_1(X)$ is residually nilpotent and that the quasi-Albanese map of $X$ is proper,
or that $X$ admits an admissible integral variation of mixed Hodge structures with a proper period map.\\

In order to both remove this properness restriction and deal with the general case, we propose the following definition.

\begin{definition*}
Let $X$ be a connected complex algebraic variety. Let $H \subset \pi_1(X)$ be a normal subgroup. The pair $(X,H)$ is called maximal if, for any connected complex algebraic variety $\bar X$ equipped with an open immersion $X \to \bar X$ and every holomorphic map $v \colon \Delta \to \bar X$ such that $v(\Delta^\ast) \subset X$ and $v(0) \notin X$, the composite homomorphism $\bZ = \pi_1(\Delta^\ast) \to \pi_1(X) \to \pi_1(X) \slash H$ is not constant. 
\end{definition*}

If $X$ is a connected smooth complex algebraic variety and $H \subset \pi_1(X)$ is a normal subgroup such that $\pi_1(X) \slash H$ is torsion-free, then the pair $(X,H)$ extends as a maximal pair on a smooth partial compactification of $X$, see Proposition \ref{maximal extension}.\\

Our main result in this paper is the following:

\begin{thmintro}\label{main theorem}
Let $X$ be a normal connected complex algebraic variety. Let $H \subset \pi_1(X)$ be a normal subgroup which is commensurable to the intersection of the kernels of the monodromy representations of a collection of semisimple complex local systems on $X$ of bounded rank. Let $\tilde{X}^H \to X$ be the Galois étale cover corresponding to $H$. Assume that the pair $(X, H)$ is maximal. Then,
\begin{enumerate}
\item  There exists a normal complex space $\tilde{\Sh}_X^H$ with no positive-dimensional compact analytic subspace and a surjective proper holomorphic map with connected fibers $\tilde{\sh}_X^H \colon \tilde{X}^H \rightarrow \tilde{\Sh}_X^H$.
\item There exists a connected complex analytic space $\Sh_X^H$ and a surjective proper holomorphic map with connected fibers $\sh_X^H \colon X \rightarrow \Sh_X^H$ such that the following property holds:\\
For any connected proper complex algebraic variety $Z$ equipped with an algebraic morphism $f \colon Z \to X$, $\sh_X^H ( f(Z) )$ is a point if and only if the image of $\pi_1(Z)$ in $\pi_1(X) \slash H$ is finite.
\end{enumerate} 
\end{thmintro}

It is easily seen that the morphisms $\tilde{\sh}_X^H$ and $\sh_X^H$ are uniquely characterized (up to unique isomorphism) by their respective defining properties.

\begin{thmintro}[{see Theorem \ref{Shafarevich_factorization}}]
Let $X$ be a normal connected complex algebraic variety. Let $\cL$ be a semisimple complex local system on $X$. Let $H$ be the kernel of the monodromy representation of $\cL$. Assume that the pair $(X, H)$ is maximal. Assume also that the image of the monodromy representation of $\cL$ is torsion-free. Then there exists a complex local system $\cM$ on $\Sh_X^H$ such that $\cL = (\sh_X^H)^{-1} \cM$. 
\end{thmintro}

When $X$ is a smooth projective variety, Theorem \ref{main theorem} was proved by Katzarkov-Ramachandran \cite{Katzarkov-Ramachandran} when $\dim X = 2$ and by Eyssidieux \cite{Eyssidieux-reductive} in general. Our proof follows a similar strategy as in \cite{Eyssidieux-reductive}, but with a twist. A crucial new ingredient is provided by the following result.

\begin{thmintro}[= Theorem {\ref{criterion of finiteness}}]
Let $\Gamma$ be a finitely generated group, $G$ be a reductive $\bQ$-algebraic group and $\Sigma$ be a $\bQ$-Zariski constructible subset of the character variety $M_B(\Gamma, G)$. Assume that for a prime number $p$ , all the conjugacy classes of representations corresponding to elements in the set $\Sigma(\bar \bQ)$ have bounded image when seen as representations with values in $G(\bar \bQ_p)$. Then $\Sigma$ consists in finitely many points. 
\end{thmintro}

\subsection*{Acknowledgements.}
I warmly thank Ben Bakker, Jeremy Daniel and Marco Maculan for interesting discussions.
  

\section{Holomorphic fibrations and Stein factorization}

A holomorphic fibration is a proper surjective holomorphic map $g \colon S \to T$ between two complex analytic spaces such that $g_\ast \cO_S = \cO_T$ (in particular, the fibres of a fibration are connected).

\begin{thm}[Stein, Cartan; cf. {\cite[Theorem 3]{Cartan}}]\label{Cartan}
Let $f_\alpha \colon X \to Y_\alpha$ be a collection of holomorphic maps between complex analytic spaces. Let $f \colon X \to  \prod_\alpha Y_\alpha$ be the product map. Assume that the connected components of the fibres of $f$ are compact. Then the equivalence relation $R$  defined by these connected components is proper and the quotient ringed space $X / R$ is a complex analytic space.
\end{thm}

The proper surjective holomorphic map with connected fibers $X \to X/R$ is called the Stein factorization of the collection of holomorphic maps ${f_\alpha }$.


\section{Maximal pairs}\label{Maximal pairs}

\begin{defn}\label{definition_maximality}
Let $X$ be a connected complex algebraic variety and $\rho \colon \pi_1(X) \to \Gamma$ a homomorphism of groups.
The pair $(X, \rho)$ is called maximal if, for any connected complex algebraic variety $\bar X$ equipped with an open immersion $X \to \bar X$ and every holomorphic map $v \colon \Delta \to \bar X$ such that $v(\Delta^\ast) \subset X$ and $v(0) \notin X$, the composite homomorphism $\bZ = \pi_1(\Delta^\ast) \to \pi_1(X) \to \Gamma$ is not constant. 
\end{defn}

If $H \subset \pi_1(X)$ is a normal subgroup, we recover the definition of the introduction by considering the homomorphism $\pi_1(X) \to \pi_1(X) \slash H$.

\begin{rem}
It is easily seen that it is sufficient to check the condition in the definition only for one compactification $\bar X$ of $X$. Also, keeping the notations of the definition, it follows immediately that if the pair $(X, \rho)$ is maximal, then for every holomorphic map $v \colon \Delta \to \bar X$ such that $v(\Delta^\ast) \subset X$ and $v(0) \notin X$, the composite homomorphism $\bZ = \pi_1(\Delta^\ast) \to \pi_1(X) \to \Gamma$ has an infinite image. 
\end{rem}

\begin{prop}\label{pull-back_maximal}
Let $f \colon Y \to X$ be a proper algebraic morphism between two connected complex algebraic varieties. Let $\rho \colon \pi_1(X) \to \Gamma$ be a homomorphism of groups. 
\begin{itemize}
\item If the pair $(X, \rho)$ is maximal, then the pair $(Y, f^{-1} \rho)$ is maximal.
\item If $f$ is surjective and the pair $(Y, f^{-1} \rho)$ is maximal, then the pair $(X, \rho)$ is maximal.
\end{itemize}
\end{prop}
\begin{proof}
Since the algebraic map $f \colon Y \to X$ is proper, there exist $\bar Y$ and $\bar X$ some connected compactifications of $Y$ and $X$ respectively, such that $f$ extends to an algebraic map $\bar f \colon \bar Y \to \bar X$ with ${\bar f}^{-1}(X) = Y$. \\

Assume that the pair $(X, \rho)$ is maximal. Let $v \colon \Delta \to \bar Y$ be a holomorphic map such that $v(\Delta^\ast) \subset Y$ and $v(0) \notin Y$. Then, the composite map $\bar f \circ v \colon \Delta \to \bar X$ is a holomorphic map such that $(\bar f \circ v)(\Delta^\ast) \subset X$ and $(\bar f \circ v)(0) \notin X$. Since the pair $(X, \rho)$ is maximal, the composite homomorphism $\bZ = \pi_1(\Delta^\ast) \to \pi_1(Y) \to \pi_1(X) \to \Gamma$ has infinite image. This shows that the pair $(Y, f^{-1} \rho)$ is maximal. \\

Conversely, assume that $f$ is surjective and that the pair $(Y, f^{-1} \rho)$ is maximal. It is harmless to assume that $\bar f$ is surjective too. Let $v \colon \Delta \to \bar X$ be a holomorphic map such that $v(\Delta^\ast) \subset X$ and $v(0) \notin X$. By shrinking $\Delta$, one may assume that $v$ is injective and extends continuoulsy to $\bar \Delta$.\\
The preimage of $v(\bar \Delta \backslash \{0\})$ by $f$ is not closed in the preimage of $v(\bar \Delta)$ by $\bar f$, since otherwise its image $f \left( f^{-1} \left( v(\bar \Delta \backslash \{0\}) \right) \right) = \bar \Delta \backslash \{0\}$ by the proper map $\bar f$ would be closed in $\bar f \left({\bar f}^{-1} \left( v(\bar \Delta) \right) \right) = \bar \Delta$. Therefore, there exists $y \in \bar Y - Y$ in the closure of the preimage of $v(\bar \Delta \backslash \{0\})$. Let $w \colon \Delta \to \bar Y$ be a holomorphic map such that $w(\Delta^\ast) \subset f^{-1}(v(\Delta^\ast))$ and $w(0) = y$. Therefore, we get a commutative diagram of holomorphic maps:
\[
\begin{tikzcd}
\Delta \arrow[d, "g"] \arrow[r, "w"] & \bar Y \arrow[d, "\bar f"] \\
\Delta \arrow[r, "v"]& \bar X
\end{tikzcd}
\]
with $g(0) = 0$. Moreover, up to shrinking both $\Delta$'s, one can assume that $g^{-1}(0) = \{0 \}$.  By assumption, the image of the homomorphism $\rho \circ f_\ast \circ (w_{|\Delta^\ast})_\ast  \colon \pi_1(\Delta^\ast) \to \Gamma$ is infinite. It follows that the image of the homomorphism $\rho \circ (v_{|\Delta^\ast})_\ast \colon \pi_1(\Delta^\ast)  \to \Gamma$ is infinite. 
\end{proof}

\begin{prop}\label{proper_period_map}
Let $X$ be a connected smooth complex algebraic variety equipped with a variation of Hodge structure whose monodromy representation $\rho$ has discrete image $\Gamma$. Let $\tilde{X}^\rho \to \cD$ be the corresponding period map. Then the following are equivalent:
\begin{enumerate}
\item The pair $(X, \rho)$ is maximal, 
\item The period map $\tilde{X}^\rho \to \cD$ is proper.
\end{enumerate}
\end{prop}

\begin{proof} 
Since Proposition \ref{pull-back_maximal} shows that it is harmless to replace $X$ by a finite étale cover, one can assume thanks to Selberg's lemma that $\Gamma$ is torsion-free. In that case, the period map $\tilde{X}^\rho \to \cD$ induces a holomorphic map on the quotients $X \to   \Gamma \backslash \cD$, and the former is the base-change of the latter along the covering map $\cD \to  \Gamma \backslash \cD$. Therefore, the period map $\tilde{X}^\rho \to \cD$ is proper if and only if the induced map $X \to   \Gamma \backslash \cD$ is proper. But the latter is proper if and only if the pair $(X,\rho)$ is maximal by Griffiths' criterion \cite[Theorem 9.5]{GriffithsIII} (see also \cite[section 2.2]{BBT-mixed-Griffiths}).
\end{proof}

\begin{prop}[Compare with {\cite[Proposition 2.4]{BBT-mixed-Griffiths}}]\label{maximal extension}
Let $X$ be a smooth algebraic variety. Let $\rho \colon \pi_1(X) \to \Gamma$ be a homomorphism of groups with torsion-free image. Then there exists $i \colon X \hookrightarrow X^\prime$ an open embedding in a smooth algebraic variety and a homomorphism $\rho^\prime \colon \pi_1(X^\prime) \to \Gamma$ such that $i ^{-1} \rho^\prime = \rho$ and the pair $(X^\prime, \rho^\prime)$ is maximal. 
\end{prop}
\begin{proof}
Let $X \subset \bar X$ be a smooth compactification such that $\bar X \backslash X = D$ is a normal crossing divisor. Fix a covering of $\bar X$ by polydisks $P \simeq \Delta^{n_p}$ such that $P^\ast := P \cap X \simeq (\Delta^\ast)^{r_P} \times \Delta^{s_P}$. For any polydisk $P$, consider the positive octant $\bR^{r_P}_ {\geq 0}$. The kernel of the restriction of $\rho$ to $\pi_1(P^\ast) = \bZ^{r_P}$ defines an integral linear subspace of $\bR^{r_P}$, and we denote by $K$ its intersection with $\bR^{r_P}_ {\geq 0}$. We may find an integral simplicial subdivision of the standard fan on $\bR^{r_P}_ {\geq 0}$ for which $K$ is
a union of facets. This subdivision corresponds to a (global) monomial modification $\bar X_P \to \bar X$ such that the restriction of $\rho$ to the preimage of $P^\ast$ is maximal, once we extend the pull-back representation over the boundary components of $\bar X_P \backslash X$ with no monodromy. Notice that any further monomial modification $\bar Z \to \bar X_P$ will also satisfy the condition above $P$. Thus, taking $\bar Z$ to be a monomial modification of $\bar X$ that dominates each of the $\bar X_P$ and extending the representation over the boundary components with no monodromy, we get a maximal representation.
\end{proof}

\section{Shafarevich morphisms}

\begin{defn}\label{H-Shafarevich morphism}
Let $X$ be a normal connected complex algebraic variety and $H \subset \pi_1(X)$ be a normal subgroup. A Shafarevich morphism for the pair $(X,H)$ is the data of a connected complex analytic space $\Sh_X^H$ and a holomorphic fibration $\sh_X^H \colon X \rightarrow \Sh_X^H$ such that the following property holds:\\
For any connected proper complex algebraic variety $Z$ equipped with an algebraic morphism $f \colon Z \to X$, $\sh_X^H ( f(Z) )$ is a point if and only if the image of $\pi_1(Z)$ in $\pi_1(X) \slash H$ is finite.
\end{defn}
Since a fibration is determined by the equivalence relation it induces on the source, a Shafarevich morphism for the pair $(X,H)$, if existing, is unique up to unique isomorphism. \\

If $\Gamma$ is a group and $\rho\colon \pi_1(X) \to \Gamma$ is a homomorphism, then a Shafarevich morphism $\sh_X^\rho \colon X \rightarrow \Sh_X^\rho$ for the pair $(X,\rho)$ is by definition a Shafarevich morphism for the pair $(X, \ker{\rho})$. Similarly, if $\Sigma = \{ \cL_i\}_{i\in I}$ is a collection of complex local systems on $X$, then
a $\Sigma$-Shafarevich morphism $\sh_X^\Sigma \colon X \rightarrow \Sh_X^\Sigma$ is by definition a Shafarevich morphism for the pair $(X, H)$, where $H$ is the intersection of the kernels of the monodromy representations of the $\cL_i$'s.

\begin{thm}\label{equivalent_Shafa}
Let $X$ be a connected normal complex algebraic variety. Let $H \subset \pi_1(X)$ be a normal subgroup and ${\tilde{X}^H} \to X$ the corresponding Galois \'etale cover. The following assertions are equivalent:
\begin{enumerate}
\item  There exists a normal complex space $\tilde{\Sh}_X^H$ with no positive-dimensional compact analytic subspace and a holomorphic fibration $\tilde{\sh}_X^H \colon \tilde{X}^H \rightarrow \tilde{\Sh}_X^H$.
\item There exists a connected complex analytic space $\Sh_X^H$ and a holomorphic fibration $\sh_X^H \colon X \rightarrow \Sh_X^H$ such that the following property holds:\\
For any connected proper complex algebraic variety $Z$ equipped with an algebraic morphism $f \colon Z \to X$, $\sh_X^H ( f(Z) )$ is a point if and only if the image of $\pi_1(Z)$ in $\pi_1(X) \slash H$ is finite.
\item There exists a connected complex analytic space $\Sh_X^H$ and a holomorphic fibration $\sh_X^H \colon X \rightarrow \Sh_X^H$ such that the following properties hold:\\
For every (connected) fiber $F$ of $\sh_X^H$, the image of $\pi_1(F)$ in $\pi_1(X) \slash H$ is finite; for every irreducible closed complex algebraic subvariety $Z \subset X$ with normalization $\bar Z \to Z$, if the image of $\pi_1(\bar Z)$ in $\pi_1(X) \slash H$ is finite, then $\sh_X^H (Z)$ is a point.
\item There exists a collection of holomorphic maps $\varphi_\alpha \colon {\tilde{X}^H} \to S_\alpha$ such that the connected components of the fibers of the map $\varphi := \prod_\alpha \varphi_\alpha \colon {\tilde{X}^H} \to \prod_{\alpha} S_\alpha$ are compact, and such that any connected compact analytic subspace of ${\tilde{X}^H}$ is contained in a fiber of $\varphi$.
\end{enumerate}
\end{thm}
The maps in 1, 2 and 3, if existing, are unique (up to unique isomorphism).

\begin{proof}
Let $\rho \colon \pi_1(X) \to \Gamma := \pi_1(X) \slash H$ be the quotient map.\\

The assertion 1 clearly implies the assertion 4. Conversely, assuming 4, let $\psi \colon {\tilde{X}^H} \to S$ be the Stein factorisation of the collection of holomorphic maps $\varphi_\alpha \colon {\tilde{X}^H} \to S_\alpha$, cf. Theorem \ref{Cartan}. Since ${\tilde{X}^H}$ is normal and $\psi$ is a proper surjective holomorphic map with connected fibers, $\psi$ is a holomorphic fibration. If $T \subset S$ is a connected compact analytic subspace, then $\psi^{-1}(T)$ is a connected compact analytic subspace of ${\tilde{X}^H}$. Therefore $\psi^{-1}(T)$ is contained in a fiber of $\varphi$, hence in a fiber of $\psi$, so that $T$ is a point. This shows 1.\\

$1 \Rightarrow 2$. By unicity of the map $\tilde{\sh}_X^H$, the Galois action of $\Gamma$ on ${\tilde{X}^H} $ descends to an action on $\tilde{\Sh}_X^H$ which is still properly discontinuous. We claim that the map induced on the quotients $\sh_X^H \colon X \rightarrow \Sh_X^H$ is the $H$-Shafarevich morphism. \\
First note that $\sh_X^H$ is proper: if $K \subset \Sh_X^H$ is a sufficiently small compact subset that lifts to $\tilde{\Sh}_X^H$, then $\left(\sh_X^H \right)^{-1}(K)$ is the image of the compact $(\tilde{\sh}_X^H)^{-1}(K)$, hence it is compact. In particular, the fibers of $\sh_X^H$ are images of the fibers of $\tilde{\sh}_X^H$, hence they are connected.\\
Let $Z$ be a connected proper complex algebraic variety equipped with an algebraic map $f \colon Z \rightarrow X$. If the image of $\pi_1(Z)$ in $\Gamma$ is finite, then the connected components of the analytic space $Z \times_X {\tilde{X}^H}$ are compact. Therefore, their image in $\tilde{\Sh}_X^H$ is a point, hence the image of $Z$ in $\Sh_X^H$ is a point.\\
Conversely, assume that the composite map $\sh_X^H \circ f \colon Z \rightarrow \Sh_X^H$ is constant, or equivalently that $f(Z)$ is contained in a fiber $F$ of $\sh_X^H$. Since $\tilde{\sh}_X^H$ is a fibration, the preimage of $F$ in ${\tilde{X}^H}$ is a disjoint union of connected compact analytic subspaces. Therefore, the image of $\pi_1(F)$ in $\Gamma$ is finite, hence so is the image of $\pi_1(Z)$ in $\Gamma$.\\ 

$ 2 \Rightarrow 3$ is immediate.\\

$3 \Rightarrow 1$.
Let $F$ be a (connected) fiber of $\sh_X^H$. Since the image of $\pi_1(F)$ in $\Gamma$ is finite, the connected components of $F \times_X {\tilde{X}^H}$ are compact. A fortiori, the connected components of the fibers of the composite map ${\tilde{X}^H} \to X \to \Sh_X^H$ are compact. Define $\tilde{\sh}_X^H \colon \tilde{X}^H \rightarrow \tilde{\Sh}_X^H$ as its Stein factorization, cf. Theorem \ref{Cartan}. To prove that $\tilde{\Sh}_X^H$ has no positive-dimensional compact analytic subspace, it is equivalent to prove that every connected compact analytic subspace $Z \subset {\tilde{X}^H}$ is contracted by $\tilde{\sh}_X^H$. One may assume that $Z$ is irreducible. Let $\bar Z \to Z$ be its normalization, so that $\bar Z$ is also a compact irreducible complex analytic space. Let $Y \subset X$ be the image of $Z$ by the map ${\tilde{X}^H} \to X $. It is a compact analytic subspace of $X$, hence it is an algebraic subvariety by applying Chow's theorem to a compactification of $X$. Let $\bar Y \to Y$ be the normalization of $Y$. Since the map $Z \to Y$ is surjective, the universal property of the normalization implies that the composite map $\bar Z \to Z \to Y$ factorizes through $\bar Y$. Therefore, one has a commutative diagram:
\[
\begin{tikzcd}
\bar Z \arrow[r] & \bar Y \times_Y {\tilde{X}^H} \arrow[d] \arrow[r] &  Z   \arrow[d] \arrow[r] & {\tilde{X}^H} \arrow[d] \\
                  & \bar Y \arrow[r]& Y \arrow[r]& X
\end{tikzcd}
\]

The compact irreducible analytic space $\bar Z$ surjects onto one of the irreducible (= connected) components of the normal analytic space $\bar Y \times_Y {\tilde{X}^H}$. Therefore, the image of $\pi_1(\bar Y)$ in $\Gamma$ is finite, so that $\sh_X^H (\bar Y)$ is a point and $\tilde{\sh}_X^H(Z)$ is a point. 
\end{proof}

\begin{prop}\label{H-Shafarevich morphism-pull-back}
Let $X$ be a connected normal complex algebraic variety and $H \subset \pi_1(X)$ be a normal subgroup. Assume that $\sh_X^H \colon X \rightarrow \Sh_X^H$ is a Shafarevich morphism for the pair $(X,H)$. Let $Y$ be a connected normal complex algebraic variety equipped with a proper morphism $f \colon Y \to X$, and $f_\ast^{-1}(H) \subset \pi_1(Y)$ be the preimage of $H$ by the induced homomorphism $f_\ast \colon \pi_1(Y) \to \pi_1(X)$. Then the  Stein factorization of the composition of the maps $Y \to X$ and $X \rightarrow \Sh_X^H$ is a Shafarevich morphism for the pair $(Y,f_\ast^{-1}(H))$.
\end{prop}
\begin{proof}
Direct consequence of the definitions.
\end{proof}

\begin{prop}\label{Shafarevich_finite_etale_map}
Let $X$ be a connected normal complex algebraic variety and $H \subset \pi_1(X)$ be a normal subgroup. Let $\pi \colon X^\prime \to X$ be a finite étale morphism, and set $H^\prime := \pi_\ast^{-1}(H) \subset  \pi_1(X^\prime) $. Then, the $(X, H)$-Shafarevich morphism exists if and only if the $(X^\prime, H^\prime)$-Shafarevich morphism exists. 
\end{prop}

\begin{proof}
If the Shafarevich morphism for the pair $(X, H)$ exists, then the Stein factorization of the composite morphism $\sh_{X}^H \circ \pi \colon X^\prime \rightarrow \Sh_{X}^{H}$ is the Shafarevich morphism of the pair $(X^\prime, H^\prime)$.\\
Let us now prove the converse. By going to a further finite étale cover of $X^\prime$ and using the first part of the proposition, one can assume that that finite étale cover $\pi \colon X^\prime \to X$ is Galois, with Galois group $\Gamma$. The Shafarevich morphism for the pair $(X^\prime, H^\prime)$ is $\Gamma$-equivariant, hence by quotienting by $\Gamma$ one gets a proper holomorphic map $X \to \Gamma \backslash  \Sh_{X^\prime}^{H^\prime}$. This map is easily seen to be the Shafarevich morphism for the pair $(X,H)$.
\end{proof}

\begin{prop}\label{Shafarevich: from smooth to normal}
Let $Y \to X$ be a surjective proper morphism with connected fibres between normal complex algebraic varieties. Let $H \subset \pi_1(X)$ be a normal subgroup. Let $G \subset \pi_1(Y)$ be the preimage of $H$ by the homomorphism of groups $\pi_1(Y) \to \pi_1(X)$. Assume that the $G$-Shafarevich morphism $\sh_{Y}^{G} \colon Y \rightarrow \Sh_{Y}^{G}$ exists. Then there is a (unique) factorization
\begin{align*}
\xymatrix{
Y \ar[r] \ar[d] &  \Sh_{Y}^{G} \\
X   \ar[ru]       &   }
\end{align*}
and the induced morphism $X \rightarrow \Sh_{Y}^{G}$ is the $H$-Shafarevich morphism.
\end{prop}

In particular, this result shows that to prove the existence of Shafarevich morphisms, one needs only to consider the case where the underlying variety is smooth quasi-projective.

\begin{proof}
Let ${\tilde{X}^H} \to X$ and ${\tilde{Y}^{G}} \to Y$ be the Galois \'etale covers corresponding to $H$ and $G$ respectively. Since $Y \to X$ is a fibration between normal varieties, the induced homomorphism between their fundamental groups is surjective, hence the following commutative diagram is cartesian:
\begin{align*}
\xymatrix{
{\tilde{Y}^{G}} \ar[r] \ar[d] &  {\tilde{X}^H} \ar[d] \\
Y   \ar[r]       & X  }
\end{align*}
In particular, the map ${\tilde{Y}^{G}} \to  {\tilde{X}^H}$ is a holomorphic fibration.
Thanks to Theorem \ref{equivalent_Shafa}, the existence of the $G$-Shafarevich morphism is equivalent to the existence of a normal complex space $\tilde{\Sh}_{Y}^{G}$ with no positive-dimensional compact analytic subspace and a holomorphic fibration $\tilde{\sh}_{Y}^{G} \colon \tilde{Y}^{G} \rightarrow \tilde{\Sh}_{Y}^{G}$.
Since the map ${\tilde{Y}^{G}} \to  {\tilde{X}^H}$ is a holomorphic fibration, it fibers are contracted by $\tilde{\sh}_{Y}^{G}$. Since ${\tilde{X}^H}$ is normal, we obtain a factorization
\begin{align*}
\xymatrix{
{\tilde{Y}^{G}} \ar[r] \ar[d] & \tilde{\Sh}_{Y}^{G} \\
{\tilde{X}^H}   \ar[ru]       &   }
\end{align*} 
Applying again Theorem \ref{equivalent_Shafa}, the existence of the holomorphic fibration $\tilde{X}^{H} \rightarrow \tilde{\Sh}_{Y}^{G}$ implies the existence of the $H$-Shafarevich morphism. Moreover, one gets the diagram from the statement by quotienting by the equivariant action of $\pi_1(X) / H = \pi_1(Y) / G$.
\end{proof}

\begin{prop}\label{Shafarevich_sum}
Let $\{ \rho_i \colon \pi_1(X) \to \Gamma_i \}_{i  \in I}$ be a finite collection of homomorphisms of groups with torsion-free image on a smooth complex algebraic variety $X$. Assume that the homomorphism $\prod_{i \in I} \rho_i \colon \pi_1(X) \to \prod_{i \in I} \Gamma_i$ is maximal on $X$. Assume moreover that for every $i$, there exists a partial smooth compactification $X_i$ of $X$ to which $\rho_i$ extends and such that the Shafarevich morphism for the pair $(X_i, \rho_i)$ exist. Then the Shafarevich morphism for the pair $(X, \prod_{i \in I} \rho_i)$ exists.
\end{prop}

\begin{proof}
First note that $\prod_{i \in I} \sh_{X_i}^{\rho_i} \colon \prod_{i \in I} X_i \to \prod_{i \in I} \Sh_{X_i}^{\rho_i}$ is the Shafarevich morphism for the representation $\prod_{i \in I} \colon \pi_1( \prod_{i \in I} X_i) \to \prod_{i \in I} \Gamma_i$. On the other hand, since $\prod_{i \in I} \rho_i$ is maximal on $X$, one has $\cap_{i \in I} X_i = X$. In particular, the diagonal embedding $X \to \prod_{i \in I} X_i$ is proper. By applying
Proposition \ref{H-Shafarevich morphism-pull-back}, it follows that the Shafarevich morphism for the pair $(X, \prod_{i \in I} \rho_i)$ exists and is equal to the Stein factorization of the composition of the maps $X \to \prod_{i \in I} X_i$ and $ \prod_{i \in I} \sh_{X_i}^{\rho_i}$.
\end{proof}

\section{Katzarkov-Zuo reductions}

\begin{defn}
Let $X$ be a connected complex algebraic variety. Let $G$ be an algebraic group defined over a non-archimedean local field $K$. Let  $\rho \colon \pi_1(X) \rightarrow G(K)$ be a representation whose image is Zariski-dense in a reductive group.\\
A Katzarkov-Zuo reduction of the pair $(X, \rho)$ is an algebraic morphism $\sigma_X^{\rho} \colon X \rightarrow S_X^{\rho}$ such that the following property holds:\\
For any connected complex algebraic variety $Z$ and any algebraic map $f \colon Z \rightarrow X$, the composite map $\sigma_X^{\rho} \circ f \colon Z \rightarrow S_X^{\rho}$ is constant if, and only if, the representation $f^{-1} \rho$ has bounded image.
\end{defn}

The following result is proved in \cite[Proposition 1.4.7]{Eyssidieux-reductive} in the proper case and in \cite{Brotbek-Daskalopoulos-Deng-Mese} and \cite[Theorem 0.6]{Cadorel-Deng-Yamanoi} in the general case.

\begin{thm}\label{Katzarkov-Zuo-smooth-case}
Let $X$ be a connected smooth quasi-projective complex algebraic variety. Let $G$ be an algebraic group defined over a non-archimedean local field $K$. Let  $\rho \colon \pi_1(X) \rightarrow G(K)$ be a representation whose image is Zariski-dense in a reductive group. Then the pair $(X, \rho)$ admits a Katzarkov-Zuo reduction.
\end{thm}

\begin{rem}
In the non-proper case, the statement of \cite[Theorem 0.6]{Cadorel-Deng-Yamanoi} is slightly weaker than the statement of Theorem \ref{Katzarkov-Zuo-smooth-case}, since $Z$ is assumed to be normal in loc. cit.. However, the proof of Theorem \ref{Katzarkov-Zuo-smooth-case} is essentially contained in the proof of Claim 5.21 in \cite{Cadorel-Deng-Yamanoi}. Namely, if $F$ is a connected component of a fiber of $\sigma_X^{\rho}$, then the arguments in loc. cit. show that every connected component of $F \times_X \tilde{X}^\rho$ is sent to a point by the pluriharmonic map with values in the building $\Delta(G)$ associated to $\rho$. This point is fixed by the image of $\pi_1(F)$ in $G(K)$. Since the stabilizer of a point in an Euclidean Bruhat-Tits building is compact, this shows that the image of $\pi_1(F)$ in $G(K)$ is bounded.
\end{rem}

\begin{prop}\label{Katzarkov-Zuo-normal-case}
Let $X$ be a connected normal complex algebraic variety. Let $G$ be an algebraic group defined over a non-archimedean local field $K$. Let  $\rho \colon \pi_1(X) \rightarrow G(K)$ be a representation whose image is Zariski-dense in a reductive group. Then the pair $(X, \rho)$ admits a Katzarkov-Zuo reduction.
\end{prop}
\begin{proof}
Let $\nu \colon X^\prime \to X$ be a surjective proper birational morphism from a connected smooth quasi-projective complex algebraic variety $X$. In particular, $\nu$ has connected fibres. Let $\rho^\prime := \nu^{-1}(\rho)$. Thanks to Theorem \ref{Katzarkov-Zuo-smooth-case}, the pair $(X^\prime, \rho^\prime)$ admits a Katzarkov-Zuo reduction $\sigma_{X^\prime}^{\rho^\prime} \colon X \rightarrow S_{X^\prime}^{\rho^\prime}$. Let us show that $\sigma_{X^\prime}^{\rho^\prime}$ factorizes through $\nu$. Since $X$ is normal, it is sufficient to prove that $\sigma_{X^\prime}^{\rho^\prime}$ is connected on every (necessarily connected) fiber $F$ of $\nu$. But this follows from the definition of a Katzarkov-Zuo reduction, since the image of $\pi_1(F)$ by $\rho^\prime$ is trivial.
\end{proof}
 
\section{Moduli spaces of representations}

\subsection{Definitions}

Let $\Gamma$ be a finitely generated group and $G$ a reductive $\bQ$-algebraic group. We denote by $R(\Gamma, G)$ the $\bQ$-affine scheme of finite type that represents the functor that associates to any $\bQ$-algebra $A$ the set of group homomorphisms from $\Gamma$ to $G(A)$. We denote by $M(\Gamma, G)$ the affine scheme corresponding to the $\bQ$-algebra $\bQ[R(\Gamma, G)]^{G}$. It is a $\bQ$-scheme of finite type. \\

If $k \supset \bQ$ is an algebraically closed field, then a representation $\rho \colon \Gamma \to G(k)$ is called reductive if the Zariski-closure of its image is a reductive subgroup of $G(k)$. The orbit under conjugation of a representation $\rho \colon \Gamma \to G(\bC)$ is closed if and only if $\rho$ is reductive. On the other hand, every fiber of the projection $R(\Gamma, G) \to  M(\Gamma, G)$ contains a unique closed orbit. Therefore, the $\bC$-points of $M(\Gamma, G)$ are naturally in bijection with the elements in $R(\Gamma, G)(\bC)$ corresponding to reductive representations. See \cite{Lubotzky-Magid} and \cite{Sikora} for details.

\subsection{The $\SL_n$-character variety}

Let $\Gamma$ be a finitely generated group and $n$ a positive integer. For any $\gamma \in \Gamma$, we denote by $tr_{\gamma}$ the character function on $R(\Gamma, \SL_n)$ defined by $tr_{\gamma}(\rho) := tr(\rho(\gamma))$. These character functions are invariant under the $\SL_n$-action by conjugation. In fact, they generate the algebra of invariant functions:

\begin{prop}
The $\bQ$-algebra $\bQ[R(\Gamma, \SL_n)]^{\SL_n}$ is generated by the functions $tr_{\gamma}, \gamma \in \Gamma$.
\end{prop}

This result is an easy consequence of a well-known result of Procesi. We give the details for the reader convenience.
\begin{proof}
Since the group $\Gamma$ is finitely generated, the affine $\bQ$-scheme $R(\Gamma, \SL_n)$ can be realized as a closed $\bQ$-subscheme of $(\SL_n)^r$ for some positive integer $r$. If $M_n$ denote the affine $\bQ$-scheme of $n \times n$ square matrices, then the group $\SL_n$ acts by conjugation on the affine $\bQ$-scheme $(M_n)^r$, and this action preserves the closed subscheme $(\SL_n)^r \subset (M_n)^r$. Therefore, the composition of the $\bQ$-algebras homomorphisms $\bQ[(M_n)^r] \rightarrow \bQ[(\SL_n)^r] \rightarrow \bQ[R(\Gamma, \SL_n)]$ is surjective. Since $\SL_n$ is reductive, the induced $\bQ$-algebras homomorphism $\bQ[(M_n)^r]^{\SL_n} \rightarrow \bQ[R(\Gamma, \SL_n)]^{\SL_n}$ is surjective too, cf. \cite[Chapter 1 \S 2]{Mumford-GIT}. Finally, our claim is a consequence of the following result of Procesi \cite{Procesi}: the $\bQ$-algebra $\bQ[(M_n)^r]^{\SL_n}$ is generated by the functions that associate to a r-tuple of matrices the trace of a monomial in these matrices.
\end{proof}

Since the $\bQ$-algebra $\bQ[R(\Gamma , \SL_n)]^{\SL_n}$ is finitely generated, one can choose finitely many elements in $\Gamma$ such that the associated character functions yield a closed embedding $M_B(\Gamma , \SL_n) \hookrightarrow \bA_{\bQ}^r$ defined over $\bQ$. 

\subsection{Representations with bounded image}
Fix a prime number $p$. We say that an element of $M_B(\Gamma , \SL_n)(\bar \bQ_p)$ has bounded image if it corresponds to a conjugacy class of (reductive) representations $ \Gamma \rightarrow \SL_n(\bar \bQ_p)$ with bounded image with respect to the topology induced by the topology of $\bar \bQ_p$.\\

Since $\Gamma$ is finitely generated, for any representation $\rho \colon \Gamma \rightarrow \SL_n(\bar \bQ_p)$, there exists a finite extension $K$ of $\bQ_p$ such that the image of $\rho$ is contained in $\SL_n(K)$. Then $\rho$ has bounded image if and only if it is conjugated to a subgroup of $\SL_n(\cO_K)$. In that case, the character functions evaluated at $\rho$ take their values in $\cO_K$.

\begin{defn}
We denote by $M(\Gamma, \SL_n)(\bar \bQ_p)^o$ the subset of $M_B(\Gamma, \SL_n)(\bar \bQ_p)$ corresponding to (conjugacy classes of) reductive representations on which all character functions take values in $\bar \bZ_p$.
\end{defn}

\begin{prop}\label{properties of bounded characters}
Fix a prime number $p$. Then:
\begin{enumerate}
\item $M(\Gamma, \SL_n)(\bar \bQ_p)^o$ is a closed subset of $M_B(\Gamma, \SL_n)(\bar \bQ_p)$ with respect to the topology induced by $\bar \bQ_p$,
\item The intersection of $M(\Gamma, \SL_n)(\bar \bQ_p)^o$ with $M_B(\Gamma, \SL_n)(K)$ is compact for every finite extension $K$ of $\bQ_p$,
\item $M(\Gamma, \SL_n)(\bar \bQ_p)^o$ contains the points that correspond to the conjugacy classes of reductive representations $\Gamma \rightarrow \SL_n(\bar \bQ_p)$ with bounded image.
\end{enumerate}
\end{prop}

\begin{rem}
One can show that an absolutely irreducible representation belongs to $M(\Gamma, \SL_n)(\bar \bQ_p)^o$ if and only if it has bounded image.
\end{rem}

\begin{proof}
Since the character functions are polynomial, the induced functions $M_B(\Gamma, \SL_n)(\bar \bQ_p) \rightarrow \bar \bQ_p$ are continuous. The first assertion follows immediately, since $\bar \bZ_p$ is closed in $\bar \bQ_p$.\\
For the second assertion, choose finitely many elements in $\Gamma$ such that the associated character functions yield a closed embedding $M_B(\Gamma , \SL_n) \hookrightarrow \bA_{\bQ}^r$ defined over $\bQ$. If $K$ is a finite extension of $\bQ_p$ and $\cO_K$ is its ring of integers, then the intersection of $M(\Gamma, \SL_n)(\bar \bQ_p)^o$ with $M_B(\Gamma, \SL_n)(K)$ coincide with the preimage of $\bA_{\bQ}^r(\cO_K)$, therefore it is compact.\\
The last assertion was proved above.
\end{proof}

\subsection{Applications}
Let $k$ be a field and $X$ be a scheme of finite type over $k$. A subset $\Sigma$ of $X$ is called constructible (with respect to the Zariski topology) if it is in the Boolean algebra generated by the closed subsets; or equivalently if $\Sigma$ is a disjoint finite union of locally closed subsets. For any field extension $l \supset k$, we denote by $\Sigma(l)$ the subset of $X(l)$ that corresponds to $k$-morphisms $\mathrm{Spec}(l) \to X$ whose image belongs to $\Sigma$. If $X \rightarrow Y$ is a $k$-morphism between two $k$-schemes of finite type and $\Sigma$ is a constructible subset of $X$, then by a classical result of Chevalley the image of $\Sigma$ is a constructible subset of $Y$.

\begin{lem} \label{a density lemma}
Let $X$ be a scheme of finite type over $\bQ$. Let $\Sigma$ be a constructible subset of $X$. If $\Sigma$ is Zariski-dense in $X$, then the set $\Sigma(\bar \bQ)$ is dense in $X(\bar \bQ_p)$ for the p-adic topology.
\end{lem}
\begin{proof}
We will prove that $\Sigma(\bar \bQ)$ is even dense in $X(\bC_p)$ for the p-adic topology. If $U \subset X$ is a smooth and dense open subset contained in $\Sigma$, then $U(\bC_p)$ is dense in $X(\bC_p)$ for the p-adic topology thanks to \cite[Proposition 3.4.4]{Berkovich}. Therefore, it is sufficient to prove that there exists such an $U$ such that $U(\bar \bQ)$ is dense in $U(\bC_p)$ for the p-adic topology. By working on each connected component of $U$, one reduces to the case where $U$ is smooth and integral. Moreover, up to shrinking $U$, one can assume that there exists a finite étale morphism $U \to V$ onto an open subset $V$ of $\bA^n$ for some $n$. Since $\bar \bQ$ is dense in $\bC_p$ for the $p$-adic topology, it is clear that $V(\bar \bQ)$ is dense in $V(\bC_p)$ for the p-adic topology, from what it follows that $U(\bar \bQ)$ is dense in $U(\bC_p)$ for the p-adic topology.
\end{proof}

\begin{lem}\label{p-adic compact affine varieties}
Let $p$ be a prime number and $M$ be an affine scheme of finite type over $\bQ_p$. If the set $M(K)$ equipped with its canonical p-adic topology is compact for every finite extension $K$ of $\bQ_p$, then $M$ has dimension zero.
\end{lem}
\begin{proof}
Since the irreducible components of $M^{red}$ satisfy the assumptions of the lemma, it is sufficient to treat the case where $M$ is integral. In that case, by Noether normalization lemma there exists a finite surjective morphism $f \colon M \rightarrow \bA_{\bQ_p}^n$ for some non-negative integer $n$. If $f$ has degree $d$, then every element of $\bA_{\bQ_p}^n(\bQ_p)$ is the image of an element of $M(K)$ for some extension $K$ of $\bQ_p$ of degree $\leq d$. Since there are finitely many extensions of $\bQ_p$ of degree $\leq d$ in a fixed algebraic closure of $\bQ_p$, by taking their compositum one gets that $\bA_{\bQ_p}^n(\bQ_p)$ is contained in the image of $M(L)$ for a finite extension $L$ of $\bQ_p$. Since $M(L)$ is compact by assumption and $f$ is proper, its image $f(M(L))$ is a compact subset of $\bA_{\bQ_p}^n(L)$. But $\bA_{\bQ_p}^n(\bQ_p)$ is closed in $\bA_{\bQ_p}^n(L)$, so that $\bA_{\bQ_p}^n(\bQ_p) = \bA_{\bQ_p}^n(\bQ_p) \cap f(M(L))$ must be compact and necessarily $n =0$.
\end{proof}

\begin{lem}\label{criterion of finiteness-Sl_n case}
Let $\Gamma$ be a finitely generated group, $n$ a positive integer and $\Sigma$ a $\bQ$-Zariski constructible subset of $M_B(\Gamma,\SL_n )$. Assume that for a prime number $p$, the set $\Sigma(\bar \bQ)$ is contained in $M(\Gamma, \SL_n)(\bar \bQ_p)^o$. Then $\Sigma$ consists in finitely many points. 
\end{lem}

\begin{proof}
Let $M$ be the $\bQ$-Zariski closure of $\Sigma$ in $M_B(\Gamma,\SL_n)$. Since $M(\Gamma, \SL_n)(\bar \bQ_p)^o$ is closed in $M_B(\Gamma, \SL_n )(\bar \bQ_p)$, it follows from Lemma \ref{a density lemma} that $M(\bar \bQ_p)$ is contained in $M(\Gamma, \SL_n)(\bar \bQ_p)^o$. Therefore, thanks to Proposition \ref{properties of bounded characters}, $M(K)$ is compact for every finite extension $K$ of $\bQ_p$, and the conclusion follows from Lemma \ref{p-adic compact affine varieties}.
\end{proof}

\begin{thm}\label{criterion of finiteness}
Let $\Gamma$ be a finitely generated group, $G$ be a reductive $\bQ$-algebraic group and $\Sigma$ be a $\bQ$-Zariski constructible subset of $M_B(\Gamma, G)$. Assume that for a prime number $p$ , all the conjugacy classes of representations corresponding to elements in the set $\Sigma(\bar \bQ)$ have bounded image when seen as representations with values in $G(\bar \bQ_p)$.
Then $\Sigma$ consists in finitely many points. 
\end{thm}
\begin{proof}
Since $G$ can be realized as a closed subgroup of $\GL_n / \bQ$ for some integer $n$ and since there is an embedding $\mathrm{GL}_n \subset  \SL_{n+1}$, $G$ is a closed subgroup of $\SL_{n+1} / \bQ$. The conjugacy class in $M_B(\Gamma , G)(\bar \bQ_p)$ of a representation $ \Gamma \rightarrow G(\bar \bQ_p)$ with bounded image is sent to the class in $M_B(\Gamma , \SL_{n+1})(\bar \bQ_p)$ of a representation $ \Gamma \rightarrow \SL_{n+1}(\bar \bQ_p)$ with bounded image through the natural map $M_B(\Gamma , G)(\bar \bQ_p) \rightarrow M_B(\Gamma, \SL_{n+1})(\bar \bQ_p)$. Since this map is finite by Proposition \ref{functoriality M_B}, it is sufficient to consider the case where $G = \SL_{n+1}$. But in that case, the assumptions imply that $\Sigma(\bar \bQ)$ is contained in $M(\Gamma, \SL_{n+1})(\bar \bQ_p)^o$ and Lemma \ref{criterion of finiteness-Sl_n case} gives the conclusion.
\end{proof}

\begin{prop}[Simpson {\cite[Corollary 9.18]{Simpson-ModuliII}}]\label{functoriality M_B}
Let $\Gamma$ be a finitely generated group and $G \rightarrow H$ be a homomorphism of complex reductive groups with finite kernel. Then the natural induced morphism $M_B(\Gamma , G) \rightarrow M_B(\Gamma, H)$ is finite.
\end{prop}


\section{The Betti moduli space}

Let $X$ be a connected complex algebraic variety. For any reductive $\bQ$-algebraic group $G$ and any $x \in X$, we set $M_B(X,x,G) := M(\pi_1(X,x), G)$. This  is an affine $\bQ$-scheme of finite type. For another choice of base-point $x^\prime \in X$, the $\bQ$-scheme $M_B(X,x^\prime,G)$ is canonically isomorphic to $M_B(X,x,G)$, hence we won't specify the base-point in the sequel. For any positive integer $n$, we let $M_B(X, n) := M_B(X, \GL_n)$.\\

When $X$ is smooth and quasi-projective, it follows from the Corlette-Simpson-Mochizuki non-abelian Hodge correspondence that $M_B(X, n)(\bC)$ is equipped with a functorial continous action of $\bC^\ast$. The points fixed by this action are exactly the isomorphism classes of complex local system underlying a complex polarized variation of Hodge structures ($\bC$-VHS). See \cite{Simpson_Higgs, Mochizuki_asterisque}.

\begin{thm}\label{Simpson}
Let $X$ be a connected smooth quasi-projective algebraic variety and $n$ a positive integer. Let $M$ be a closed algebraic subset of $M_B(X, n)$. If its set of complex points $M(\bC)$ is $\bC^\ast$-invariant, then it contains a point that corresponds to a complex local system underlying a $\bC$-VHS. 
\end{thm}
In particular, every rigid semisimple local system underlies a $\bC$-VHS.

\begin{proof}
Thanks to Mochizuki, every complex local system can be deformed continuously to a $\bC$-VHS \cite[Theorem 10.5]{Mochizuki_asterisque}. More precisely, he proves the following: if $\rho \in M_B(X, n)(\bC)$, then $\Delta^\ast \cdot \rho$ is a relatively compact subset of $ M_B(X, n)(\bC)$, so that there exists a sequence $t_i \in \Delta^\ast$ converging to zero such that $\rho_1 = \lim_{i \to \infty} t_i \cdot \rho$ exists in $ M_B(X, n)(\bC)$. Moreover, after iterating this procedure a finite number of times, one gets a representation which is fixed by the $\bC^\ast$-action. Clearly, if one starts with $\rho \in M(\bC)$, the constructed $\bC$-VHS belongs to $M(\bC)$ too.
\end{proof}

\begin{prop}\label{locus-maximal}
Let $X$ be a connected algebraic variety and $G$ a reductive $\bQ$-group.
Let $\Sigma$ be a subset of $M_B(X, G)(\bC)$. Let $M \subset M_B(X, G)$ be the $\bQ$-Zariski closure of $\Sigma$, i.e. the smallest closed $\bQ$-subscheme of $M_B(X, G)$ whose set of complex points contains $\Sigma$. Then, the pair $(X, \Sigma)$ is maximal if and only if the pair $(X, M)$ is maximal.
\end{prop}
\begin{proof}
Let $\bar X$ be a connected complex algebraic variety $\bar X$ equipped with an open immersion $X \to \bar X$. Let $v \colon \Delta \to \bar X$ be a holomorphic map such that $v(\Delta^\ast) \subset X$ and $v(0) \notin X$. Let $x := v(\frac{1}{2})$. The continuous map $v_{|\Delta^\ast}$ induces a homomorphism of groups $\bZ = \pi_1(\Delta^\ast, \frac{1}{2}) \to \pi_1(X, x)$ and a $G$-equivariant $\bQ$-algebraic map 
\[R(\pi_1(X,x), G) \to R(\bZ, G) = G.\]
Its fiber at $1_G \in G$ is a $G$-invariant closed $\bQ$-subscheme of $R(\pi_1(X,x), G)$, hence it is the preimage of a closed $\bQ$-subscheme $M_v \subset M_B(X, G)$, cf. \cite[Theorem 1.1]{Mumford-GIT}.\\
Let us now proceed to the proof of the proposition. If $(X, \Sigma)$ is maximal, then a fortiori the pair $(X, M)$ is maximal. If $(X, \Sigma)$ is not maximal, then by definition there exists a holomorphic map $v \colon \Delta \to \bar X$ as above such that $\Sigma \subset M_v(\bC)$. Since $M_v$ is a closed $\bQ$-subscheme of $M_B(X, G)$, it follows from the definition of $M$ that $M \subset M_v$. Therefore, the pair $(X, M)$ is not maximal.
\end{proof}
 
\section{A generalization of a result of Lasell and Ramachandran}

Let $X$ be a connected smooth quasi-projective variety. Let $Z$ be a smooth projective (but non-necessarily connected) variety equipped with a morphism $Z \to X$. We assume that the image $Y$ of $Z \to X$ is a  connected proper variety. (For example, $Y$ could be a connected proper subvariety of $X$ and $Z \to Y$ be the disjoint union of a desingularization of every irreducible components of $Z$.) We will keep these notations for the remaining of this section.

\begin{lem}[{Compare with \cite[Lemma 2.1]{Lasell-Ramachandran}}]\label{Lemma_Lasell-Ramachandran}
\begin{enumerate}
\item Let $\rho \colon \pi_1(X) \to \GL_n(\bC)$ be a reductive representation. Assume that for every connected component $Z_i$ of $Z$, the induced representation $\rho_{Z_i} \colon \pi_1(Z_i) \to \GL_n(\bC)$ has finite image. Then $\rho_{Y} \colon \pi_1(Y) \to \GL_n(\bC)$ has bounded image.
\item  Let $p$ be a prime number. Let $\rho \colon \pi_1(X) \to \GL_n(\bar \bQ_p)$ be a reductive representation. Assume that for every connected component $Z_i$ of $Z$, the induced representation $\rho_{Z_i} \colon \pi_1(Z_i) \to \GL_n(\bar \bQ_p)$ has bounded image. Then $\rho_{Y} \colon \pi_1(Y) \to \GL_n(\bar \bQ_p)$ has bounded image.
\end{enumerate}
\end{lem}

\begin{proof}
\begin{enumerate}
\item Let $\Omega$ denote the symmetric space of maximal compact subgroups of $\GL_n(\bC)$, on which $\GL_n(\bC)$ acts by conjugation. Thanks to Mochizuki \cite[Theorem 25.28]{Mochizuki}, there exists a $\rho$-equivariant pluriharmonic map $f \colon \tilde{X}^\rho \to \Omega$. By assumption, the irreducible components of $Z_i \times_X \tilde{X}^\rho$ are compact, therefore each of them is sent to a point in $\Omega$  by the pluriharmonic map $f$. It follows that the connected components of $Y \times_X \tilde{X}^\rho$ are also to a point in $\Omega$  by the pluriharmonic map $f$. Moreover, due to the $\rho$-equivariance of $f$, this point must be a fixed point for the induced action of $\pi_1(Y)$ on $\Omega$. This shows that $\rho_Y$ is bounded.
\item Let $\sigma_X^{\rho} \colon X \rightarrow S_X^{\rho}$  be a Katzarkov-Zuo reduction of the pair $(X, \rho)$, which exists thanks to Theorem \ref{Katzarkov-Zuo-smooth-case}. By assumption, the composite map $Z_i \to X \to S_X^{\rho}$ is constant for every $i$, hence the composite map $Y  \to X \to S_X^{\rho}$ is also constant.  This implies in turn that $\rho_{Y}$ has bounded image.
\end{enumerate}
\end{proof}

\begin{thm}[{Compare with \cite[Theorem 4.1]{Lasell-Ramachandran}}]\label{Theorem_Lasell-Ramachandran}
For every positive integer $n$, there exists a finite quotient $\Delta_n$ of $\pi_1(Y)$ such, if $\rho \colon \pi_1(X) \to \GL_n(\bC)$  is a reductive representation whose pull-back to every connected component of $Z$ is the trivial representation, then the induced representation $\rho_{Y} \colon \pi_1(Y) \to \GL_n(\bC)$ factorizes through $\Delta_n$.
\end{thm}

\begin{proof}
By Lemma \ref{Lemma_Lasell-Ramachandran}, if $\rho \colon \pi_1(X) \to \GL_n(\bC)$ is a reductive representation whose pull-back to every connected component $Z_i$ of $Z$ is the trivial representation, then $\rho_{Y} \colon \pi_1(Y) \to \GL_n(\bC)$ has bounded image. In particular, $\rho_{Y} $ is a reductive representation. Let $M \subset M_B(X, n)$ be the intersection of the fibers at the trivial representation of the morphisms $M_B(X, n ) \to M_B(Z_i, n)$. Therefore, $M$ is a closed subscheme of $M_B(X,n)$, whose $\bC$-points parametrize the conjugacy classes of reductive representations $\rho \colon \pi_1(X) \to \GL_n(\bC)$ that pull-back to the trivial representation on every $Z_i$. Let $N$ denote the image of $M$ by the algebraic morphism $M_B(X, n ) \to M_B(Y, n)$. By Chevalley's Theorem, it is a constructible subset of $M_B(Y,n)$. A priori, $N(\bC)$ is the set of conjugacy classes of semisimplifications of representations pulled-back from $M(\bC)$. However, it follows from the previous paragraph that $N(\bC)$ is in fact exactly the set of conjugacy classes of representations of $\pi_1(Y)$ pulled-back from by $M(\bC)$. On the other hand, it follows from Lemma \ref{Lemma_Lasell-Ramachandran} that for every prime $p$ one has $N(\bar \bQ_p) \subset M_B(Y,n)(\bar \bQ_p)^o$. Therefore, thanks to Theorem \ref{criterion of finiteness}, $N$ is a closed subscheme of $M_B(Y,n)$ of dimension zero. In particular, $N(\bC)$ is a finite set, showing that there are only finitely many conjugacy classes of reductive representations $\rho \colon \pi_1(X) \to \GL_n(\bC)$ that pull-back to the trivial representation on every $Z_i$. Finally, let ${\{ \rho_i \colon \pi_1(Y) \rightarrow \GL_n(\bar \bQ) \}}_{i = 1 \ldots r}$ be a finite set of reductive representations stable by $\mathrm{Gal}(\bar \bQ \slash \bQ)$-conjugation that lifts the finitely many points of $N(\bC) = N(\bar \bQ)$. Let $\rho_N \colon \pi_1(Y) \to \GL_{nr}(\bar \bQ)$ be the sum of the $\rho_i$. It follows from the $\mathrm{Gal}(\bar \bQ \slash \bQ)$-invariance that $\rho_N$ takes its values in $ \GL_{nr}(\bQ)$. Moreover, thanks to Lemma \ref{Lemma_Lasell-Ramachandran}, the image of $\rho_N$ is bounded and conjugated to a subgroup of $ \GL_{nr}(\bZ)$. Therefore, the image $\Delta_n$ of $\rho_N$ is the searched finite quotient of $\pi_1(Y)$.
\end{proof}

\begin{cor}\label{Corollary_Lasell-Ramachandran}
For every positive integer $n$, let $H_n$ be the intersection of the kernels of all reductive representations $\rho \colon \pi_1(X) \to \GL_n(\bC)$ that pull-back to the trivial representation on every connected component of $Z$. Then the image of $\pi_1(Y)$ in $ \pi_1(X) \slash H_n$ is finite.
\end{cor}


\section{Proof of Theorem \ref{main theorem}}

\subsection{A special case}\label{A special case}
Let $X$ be a connected smooth quasi-projective algebraic variety. Let $n$ be a positive integer. Let $M \subset M_B(X, n)$ be a closed algebraic subvariety defined over $\bQ$, whose set of complex points $M(\bC)$ is invariant under the $\bC^\ast$-action on $M_B(X, n)(\bC)$. \\

Assuming that the pair $(X, M)$ is maximal, we prove in the sequel the existence of the $M$-Shafarevich morphism.\\

For every $\rho \in M(\bar \bQ)$ and every prime number $p$, let $\rho_p$ denote the image of $\rho$ by the inclusion $M(\bar \bQ) \subset M(\bar \bQ_p)$, and let $ \sigma_X^{\rho_p} \colon X \rightarrow S_X^{\rho_p}$ be a Katzarkov-Zuo reduction of $(X, \rho_p)$. We denote by $\sigma_X^{M} \colon X \rightarrow S_X^{M}$ the map $X \to \prod_{\rho \in M(\bar \bQ)} \prod_{\text{p prime}} S_X^{\rho_p}$. Note that the fibers $\sigma_X^{M}$ are algebraic (since they are the possibly infinite intersection of fibers of algebraic maps). In fact, by Noetherianity, the map obtained by composing $\sigma_X^{M}$ with the projection onto finitely many well-choosen factors will have the same fibers as $\sigma_X^{M}$. \\

Let $M(\bC)^{\mathrm{VHS}}$ denote the subset of $M(\bC)$ composed by the fixed points for the $\bC^\ast$-action. It is non-empty thanks to Theorem \ref{Simpson}. Fix a point $x \in X$, and for every element of $M(\bC)^{\mathrm{VHS}}$, choose a representative $\rho \colon \pi_1(X,x) \rightarrow \GL_n(\bC) $. Let $W_{\rho}/ \bR$ be the real Zariski closure of the image of $\pi_1(X,x)$, so that $\rho$ factorizes as $\rho \colon \pi_1(X,x) \rightarrow W_{\rho}(\bR) \subset \GL_n(\bC)$. By assumption, the real Lie group $W_{\rho}(\bR)$ acts transitively on a period domain $\cD_\rho$, so that the stabilizer of a point in $\cD_{\rho}$ is a compact subgroup, and there exists a period map $\tilde X \rightarrow \cD_{\rho} $ which is equivariant with respect to $\rho \colon \pi_1(X,x) \rightarrow W_{\rho}(\bR) $. By taking their products, we obtain a map $\tilde X \rightarrow \prod_{[\rho] \in M(\bC)^{\mathrm{VHS}}} \cD_{\rho} $, which is equivariant with respect to the homomorphism $ \pi_1(X,x) \rightarrow \prod_{[\rho] \in M(\bC)^{\mathrm{VHS}}} W_{\rho}(\bR) $.\\

Finally, putting the archimedean places and the non-archimedean places together, we get a map $ \tilde X \rightarrow \prod_{[\rho] \in M(\bC)^{\mathrm{VHS}}} \cD_{\rho} \times S_X^{M} $. Letting $H := \cap_{[\rho] \in M(\bC)}  \ker{\rho}$ and ${\tilde{X}^M} := \tilde{X}^H$, the preceding map factorizes through $\tilde X \rightarrow {\tilde{X}^M}$, so that we obtain a map 
\[ \phi_{(X,x)}^M \colon {\tilde{X}^M} \rightarrow \prod_{[\rho] \in M(\bC)^{\mathrm{VHS}}} \cD_{\rho}  \times S_X^{M}.\]
 Observe that for another choice of base point $x ^\prime \in X$, the choice of a (homotopy class of a) path from $x$ to $x^\prime$ induces an isomorphism between $\pi_1(X,x)$ and $\pi_1(X, x^\prime)$, and an automorphism of ${\tilde{X}^M}$ over $X$, such that the $\pi_1(X,x)$-equivariant map $\phi_{(X,x)}$ gets identified with the $\pi_1(X, x^\prime)$-equivariant map $\phi_{(X,x^\prime)}$. \\

In view of Theorem \ref{equivalent_Shafa}, Theorem \ref{main theorem} is a consequence of the following result. 
\begin{thm}\label{key theorem}
Consider the map 
\[ \phi_X^M \colon {\tilde{X}^M} \rightarrow \prod_{[\rho] \in M(\bC)^{\mathrm{VHS}}} \cD_{\rho}  \times S_X^{M} .\]
Then:
\begin{enumerate}
\item Every connected compact analytic subspace of ${\tilde{X}^M}$ is contained in a fibre of $\phi_X^M$.
\item The connected components of the fibres of $\phi_X^M$ are compact. 
\end{enumerate}
\end{thm}

\begin{proof}[Proof of Theorem \ref{key theorem}]
We start by proving the first assertion. Let $Y$ be a connected compact analytic subspace of ${\tilde{X}^M}$. It is harmless to assume that $Y$ is irreducible. Let $f \colon Y \to X$ be the holomorphic map defined as the composition of the inclusion $Y \subset {\tilde{X}^M}$ and the projection ${\tilde{X}^M} \to X$. Then $f$ is proper (since $Y$ is compact) and has discrete fibres, hence it is a finite holomorphic map. Its image $Z := f(Y)$ is a compact closed analytic subspace of the algebraic variety $X$, hence it is a closed algebraic subvariety of $X$ by applying Chow's theorem to an algebraic compactification of $X$. By GAGA, there exists a unique algebraic structure on $Y$ such that the map $Y \to Z$ is algebraic.\\
Let $W \to Y$ be a desingularization of $Y$. Since by definition the map $f \colon Y \to X$ factorizes through ${\tilde{X}^M}$, the image of $M(\bC)$ under the natural algebraic map $M_B(X, n) \rightarrow M_B(W, n)$ is trivial. It follows that for every (conjugacy class of) representation $\rho \in M(\bar \bQ)$ and every prime number $p$, the Katzarkov-Zuo reduction of $(X, \rho_p)$ is constant on $W$, hence on $Y$. Moreover, since every period map with trivial monodromy on a connected compact complex manifold is constant, the first assertion follows. \\

Let us now prove the second assertion. Let $F$ be a connected component of a fibre of $\phi_X^M$; we want to prove that $F$ is compact. Observe that $F$ is a fortiori contained in a fibre of ${\tilde{X}^M} \rightarrow S_X^{M}$. In fact, since $F$ is connected, it is contained in the preimage of a connected component $Y$ of a fibre of $X \rightarrow S_X^{M}$. Note that $Y$ is a connected closed algebraic subvariety of $X$. As observed before, we can freely move the base point $x \in X$ and assume that $x \in Y$.\\

Let $\{Y_i\}$ denote the irreducible components of $Y$. For every $i$, let $Z_i \to Y_i$ be a proper desingularization of $Y_i$. Let $N_i$ be the image of $M(\bar \bQ)$ under the natural algebraic map $M_B(X, n) \rightarrow M_B(Z_i, n)$. It is a $\bQ$-constructible subset of $M_B(Z_i,n)$ by Chevalley's Theorem. Moreover, by definition of $Y$, the representations corresponding to points in $N_i$ have bounded image when seen as representations with values in $\GL_n(\bar \bQ_p)$ for every prime $p$. Therefore, thanks to Theorem \ref{criterion of finiteness}, the $\bQ$-constructible set $N_i$ has dimension zero. It follows that there exists a number field $K \subset \bar \bQ$ and a finite set of representations ${\{ \rho^j_i \colon \pi_1(Z_i) \rightarrow \GL_n(K) \}}_{j = 1 \ldots r_i}$ stable by $\mathrm{Gal}(\bar \bQ \slash \bQ)$-conjugation that lifts the finitely many points of $N_i$.\\

Every point of $N_i$ is the image of at least one connected component of $M(\bC)$. Since $M(\bC)$ is closed in $M_B(X, n)(\bC)$ and $\bC^\ast$-invariant, the connected components of $M(\bC)$ are also closed in $M_B(X, n)(\bC)$ and $\bC^\ast$-invariant (since $\bC^\ast$ is connected). It follows from Theorem \ref{Simpson} that they all contain a $\bC^\ast$-fixed point. Therefore, for every $\rho^j_i$, there exists a representation $\prescript{X}{}{\rho}^j_i \colon \pi_1(X) \rightarrow \GL_n(\bC)$ underlying a polarized complex variation of Hodge structures such that the composition with $\pi_1(Z_i) \rightarrow \pi_1(X)$ is equal to the composition of $\rho^j_i \colon \pi_1(Z_i) \rightarrow \GL_n(K)$ with the inclusion $\GL_n(K) \subset \GL_n(\bC)$.\\

The image $\Gamma_i$ of the representation $ \prod_{j= 1}^{r_i} \rho_i^j$ is discrete. Indeed, it is contained in $\prod_{j= 1}^{r_i} \GL_n(K)$, stable by $\mathrm{Gal}(\bar \bQ \slash \bQ)$-conjugation and bounded at every non-archimedean places. Therefore, through the canonical inclusion $\prod_{j= 1}^{r_i} \GL_n(\bC) \subset \GL_{n r_i}(\bC)$, it is contained in a subgroup conjugated to $\GL_{nr_i}(\bZ)$ \cite{Bass}. Hence we can form the following commutative diagram:
\begin{align*}
\xymatrix{
Z_i \times_X \tilde{X}^M \ar[r] \ar[d] &  \prod_{j= 1}^{r_i} \cD_{\prescript{X}{}{\rho}^j_i }  \ar[d] \\
Z_i  \ar[r]       & \Gamma_i \backslash \left(\prod_{j= 1}^{r_i} \cD_{\prescript{X}{}{\rho}^j_i }  \right) } 
\end{align*}

Let $W_i$ denote the fibre of the map $Z_i  \to  \Gamma_i \backslash \left(\prod_{j= 1}^{r_i} \cD_{\prescript{X}{}{\rho}^j_i }  \right)$ that contains the projection of  $F \times_{ \tilde{X}^M} Z_i \subset Z_i \times_X \tilde{X}^M$. Since the pair $(Z_i, N_i)$ is maximal thanks to Proposition \ref{pull-back_maximal}, $W_i$ is proper thanks to Griffiths' criterion \cite[Theorem 9.5]{GriffithsIII}. By construction, the image of $\pi_1(W_i)$ in $\pi_1(X) \slash H_M$ is finite. Let $T_i$ be a (non-necessarily connected) smooth projective variety equipped with a surjective morphism $T_i \to W_i$ such that the image of $\pi_1(T_i)$ in $\pi_1(X) \slash H_M$ is trivial. \\

Let $W \subset Y$ be  the union of the images of the $W_i$'s. It is a proper subvariety of $Y$ that contains the projection of $F$ in $X$. Up to replacing $W$ by its connected component that contains the projection of $F$, on can assume that $W$ is connected. Thanks to Corollary \ref{Corollary_Lasell-Ramachandran}, it follows that the image of $\pi_1(W)$ in $\pi_1(X) \slash H_M$ is finite. Therefore, the connected components of $W \times_X \tilde{X}^H$ are compact. Since $F$ is closed in one of the latter, it follows that $F$ is compact.
\end{proof}

\subsection{The general case}\label{The general case}

We start with a result of independant interest. Its proof is a direct consequence of Lemma \ref{finite_motivic} below.

\begin{prop}\label{Shafa-reduction-motivic}
Let $X$ be a connected smooth quasi-projective algebraic variety. Let $n$ be a positive integer. Let $\Sigma$ be a subset of $M_B(X, n)(\bC)$. Let $M \subset M_B(X, n)$ be the smallest closed subscheme defined over $\bQ$, whose set of complex points $M(\bC)$ contain $\Sigma$ and is $\bC^\ast$-invariant. Then, the $\Sigma$-Shafarevich morphism $\sh_X^\Sigma$ exists if and only if the $M$-Shafarevich morphism $\sh_X^M$ exists. In that case, $\sh_X^\Sigma = \sh_X^M$.
\end{prop}

\begin{lem}\label{trivial_motivic}
Let $Y$ be a connected smooth projective variety equipped with a morphism $Y  \to X$. Then, the image of $\pi_1(Y)$ in $\pi_1(X) \slash H_\Sigma$ is trivial if and only if the image of $\pi_1(Y)$ in $\pi_1(X) \slash H_M$ is trivial.
\end{lem}
\begin{proof}
Assume that the image of $\pi_1(Y)$ in $\pi_1(X) \slash H_\Sigma$ is trivial. Let $N$ be the preimage of the trivial representation by the algebraic morphism $M_B(X, n) \to M_B(Y, n)$. Then $N$ is a closed subscheme of $M_B(X, n)$, whose set of complex points $N(\bC)$ contain $\Sigma$ and is $\bC^\ast$-invariant. Therefore, $M \subset N$. In other words, the image of $\pi_1(Y)$ in $\pi_1(X) \slash H_M$ is trivial. Since $H_M \subset H_\Sigma$, the converse is obvious.
\end{proof}

\begin{lem}\label{finite_motivic}
Let $Z$ be a connected proper variety equipped with a morphism $Z  \to X$. Then, the image of $\pi_1(Z)$ in $\pi_1(X) \slash H_\Sigma$ is finite if and only if the image of $\pi_1(Z)$ in $\pi_1(X) \slash H_M$ is finite.
\end{lem}
\begin{proof}
If the image of $\pi_1(Z)$ in $\pi_1(X) \slash H_M$ is finite, since $H_M \subset H_ \Sigma$, the image of $\pi_1(Z)$ in $\pi_1(X) \slash H_{\Sigma}$ is a fortiori finite. \\
Conversely, assume that the image of $\pi_1(Z)$ in $\pi_1(X) \slash H_\Sigma$ is finite. Let $Z^\prime \to Z$ be a connected finite étale cover such that the image of $\pi_1(Z^\prime)$ in $\pi_1(X) \slash H_\Sigma$ is trivial. Let $Y \to Z^\prime$ be a proper desingularization of an irreducible component of $Z^\prime$. Then the image of $\pi_1(Y)$ in $\pi_1(X) \slash H_\Sigma$ is trivial. It follows by the Lemma \ref{trivial_motivic} that the image of $\pi_1(Y)$ in $\pi_1(X) \slash H_M$ is trivial. Finally, thanks to Corollary \ref{Corollary_Lasell-Ramachandran}, the image of $\pi_1(Z)$ in $\pi_1(X) \slash H_M$ is finite.
\end{proof}

Let $X$ be a normal connected complex algebraic variety. Let $H \subset \pi_1(X)$ be a normal subgroup which is commensurable to the intersection of the kernels of the monodromy representations of a collection of semisimple complex local systems $\cL_i$ on $X$ of bounded rank. Assume that the pair $(X, H)$ is maximal. In the remaining of this section, we prove the existence of the $H$-Shafarevich morphism, therefore ending the proof of Theorem \ref{main theorem}.\\

Thanks to Proposition \ref{pull-back_maximal} and Proposition \ref{Shafarevich_finite_etale_map}, we can assume that $H$ is equal to the intersection of the kernels of the monodromy representations of the $\cL_i$'s. Let $X^\prime \to X$ be a birational fibration from a smooth quasi-projective variety and let $H^\prime \subset \pi_1(X^\prime)$ be the intersection of the kernels of the monodromy representations of the pull-back of the $\cL_i$'s on $X^\prime$. Thanks to Proposition \ref{pull-back_maximal} and Proposition \ref{Shafarevich: from smooth to normal}, it is sufficient to prove the existence of the Shafarevich morphism for the pair $(X^\prime, H^\prime)$. Therefore, we can assume from the beginning that $X$ is smooth quasi-projective. Up to adding some trivial local systems to some of the $\cL_i$'s, one can assume in addition that all the $\cL_i$'s have the same rank $n \geq 1$.  Let $\Sigma \subset M_B(X, n)(\bC)$ be the subset of the corresponding conjugacy classes of monodromy representations. Let $M \subset M_B(X, n)$ be the smallest closed $\bQ$-subscheme whose set of complex points $M(\bC)$ contain $\Sigma$ and is $\bC^\ast$-invariant. Since the pair $(X, \Sigma)$ is maximal by assumption, the pair $(X, M)$ is maximal a fortiori. Therefore, the results of section \ref{A special case} implies the existence of the $M$-Shafarevich morphism. Finally, by Proposition \ref{Shafa-reduction-motivic}, the $M$-Shafarevich morphism is also a $\Sigma$-Shafarevich morphism.

\section{Additional results}

\begin{prop}\label{Reduction to one local system}
Let $X$ be a normal connected complex algebraic variety. Let $\Sigma$ be collection of semisimple complex local systems of bounded rank on $X$ and let $\sh_X^\Sigma$  be the corresponding $\Sigma$-Shafarevich morphism. Then there exists a semisimple complex local system $\cL$ on $X$ such that the $\cL$-Shafarevich morphism is equal to $\sh_X^\Sigma$.
\end{prop}
\begin{proof}
Arguing as in section \ref{The general case}, one can assume in addition that $X$ is smooth quasi-projective and that all the elements in $\Sigma$ have same rank $n \geq 1$. Let $M$ be the smallest closed subscheme of $M_B(X, n)$ defined over $\bQ$, whose set of complex points $M(\bC)$ contain $\Sigma$ and is $\bC^\ast$-invariant. Let $\Phi \subset M(\bC)$ be a finite subset which is $\bQ$-Zariski-dense. Then, thanks to Proposition \ref{Shafa-reduction-motivic}, the pair $(X, \Phi)$ is maximal. Moreover, the $\Phi$-Shafarevich morphism and the $\Sigma$-Shafarevich morphism are both equal to the $M$-Shafarevich morphism. Therefore, one can take for $\cL$ the sum of one local system in each conjugacy classes that belong to $\Phi$.
\end{proof}

The following result seems to have been unnoticed already in the compact case.
\begin{thm}\label{Shafarevich_factorization}
Let $X$ be a normal connected complex algebraic variety. Let $\rho \colon \pi_1(X) \to \Gamma$ be a surjective homomorphisms of groups with torsion-free image. Assume that the Shafarevich morphism $\sh_X^{\rho} \colon X \rightarrow \Sh_X^{\rho}$ for the pair $(X, \rho)$ exists. Then the representation $\rho$ factorizes through the homomorphism $\pi_1(X) \to \pi_1(\Sh_X^{\rho})$.
\end{thm}
\begin{proof}
Let $\tilde{X}^\rho \to X$ be the Galois étale cover associated to $\rho$, equipped with its $\Gamma$-action. 
Thanks to Proposition \ref{equivalent_Shafa} and its proof, there exists a holomorphic fibration ${\tilde{X}^\rho} \to S$, with $S$ a normal complex analytic space which does not have any positive-dimensional compact analytic subspace. By unicity, the Galois action of $\Gamma$ on ${\tilde{X}^\rho} $ descends to an action on $S$ which is still properly discontinuous. Moreover, the map induced on the quotients $\sh_X^\rho \colon X \rightarrow \Sh_X^\rho$ is the $\rho$-Shafarevich morphism. Since $\Gamma$ is torsion-free and the action of $\Gamma$ on $S$ is proper, the action of $\Gamma$ on $S$ is also free. Therefore, the map $S \to  \Sh_X^\rho$ is a Galois étale cover with Galois group $\Gamma$. 
To conclude the proof, we need to check that the following commutative diagram is cartesian:
\[
\begin{tikzcd}
\tilde{X}^\rho \arrow[d] \arrow[r] & S \arrow[d] \\
X \arrow[r, "\sh_X^\rho"]&  \Sh_X^\rho
\end{tikzcd}
\]
But, in the commutative diagram 
\[
\begin{tikzcd}
\tilde{X}^\rho  \arrow[r] &  X \times_{\Sh_X^\rho} S \arrow[r] \arrow[d]  &  S \arrow[d] \\
  &  X \arrow[r, "\sh_X^\rho"]&  \Sh_X^\rho,
\end{tikzcd}
\]
the map $\tilde{X}^\rho  \to X \times_{\Sh_X^\rho} S$ is a $\Gamma$-equivariant étale map between two connected $\Gamma$-Galois étale covers of $X$, hence it is an isomorphism.
\end{proof}

\section{The abelian case}
In this section, we describe an alternative construction of the Shafarevich morphism associated to a representation whose image is abelian. \\

Let $X$ be a connected smooth algebraic variety. Let $\rho \colon \pi_1(X) \to \Gamma$ be a homomorphism of groups, with $\Gamma$ abelian. Up to replacing $\Gamma$ by the quotient of the image of $\rho$ by its torsion subgroup, one can assume that $\Gamma$ is a free $\bZ$-module. Since $\Gamma$ is abelian, $\rho$ factorizes through the abelianization $\pi_1(X) \to \mathrm{H}_1(X , \bZ)$ of $\pi_1(X)$. Let $\alb \colon X \to \Alb(X)$ denotes the Albanese mapping, where $\Alb(X)$ is a semi-abelian variety (cf. \cite{Fujino-Albanese}). If $\bar X$ is a smooth compactification of $X$ such that $D := \bar X - X$ is a simple normal crossing divisor, then $\Alb(X) := \HH^0(\bar X, \Omega^1_{\bar X}(\log D))^\vee \slash \HH_1(X, \bZ)$. Let $\Alb(X) \to A_\rho$ be the quotient of $\Alb(X)$ by the biggest semi-abelian subvariety $T$ of $\Alb(X)$ such that $H_1(T, \bZ ) \subset \ker(H_1(X, \bZ) \to \Gamma)$.

\begin{thm}\label{Shafarevich in the abelian case}
If the pair $(X, \rho)$ is maximal, then the composition of the Albanese morphism $ X \to \Alb(X)$ with  $\Alb(X) \to A_\rho$ is a proper morphism, and its Stein factorization is the $\rho$-Shafarevich morphism.
\end{thm}
\begin{proof}
Let $K$ be the biggest sub-$\bZ$-mixed Hodge structure of $H_1(X, \bZ)$ such that the underlying $\bZ$-module is contained in the kernel of the morphism $H_1(X, \bZ) \to \Gamma$ induced by $\rho$. Let $H$ be the quotient of $H_1(X, \bZ)$ by $K$, so that $H$ is a $\bZ$-mixed Hodge structure, whose underlying $\bZ$-module is torsion-free.
Since the pair $(X, \rho)$ is maximal, the homomorphism $\pi_1(X) \to H_1(X, \bZ) \to H$ is also maximal. Moreover, for every connected complex algebraic variety $Z$ equipped with a morphism $Z \to X$, the induced homomorphism $\pi_1(Z) \to \Gamma$ factorizes as $\pi_1(Z) \to H_1(Z, \bZ) \to H_1(X, \bZ) \to \Gamma$. Therefore, the homomorphism $\pi_1(Z) \to \Gamma$ is zero if and only if the homomorphism $\pi_1(Z) \to H$ is zero. In other words, the $\rho$-Shafarevich morphism exists if and only if the Shafarevich morphism of the homomorphism $\pi_1(Z) \to H$ exists. Therefore, one can assume that $\Gamma$ is the $\bZ$-module of a quotient $\bZ$-mixed Hodge structure $H_1(X, \bZ)$.\\

Fix $x \in X$. For every $y \in X$ distinct from $x$, there is an exact sequence of $\bZ$-mixed Hodge structures:
\[  0 \to H_1(X, \bZ) \to H_1(X, \{x,y \}, \bZ) \to H_0(\{x,y\}, \bZ) \to H_0(X, \bZ) .\]
The kernel of morphism $H_0(\{x, y\}, \bZ) \to H_0(X, \bZ)$ is isomorphic to $\bZ(0)$, by sending $1$ to $y -x$. Therefore, the preceding exact sequence yields the exact sequence:
\begin{equation}\label{Albanese_exact_sequence}
0 \to H_1(X, \bZ) \to H_1(X, \{x,y \}, \bZ) \to \bZ(0) \to 0.
\end{equation}

This provides a map 
\[a \colon X - \{x\} \to \mathrm{Ext}^1_{\bZ-MHS}(\bZ(0) , H_1(X, \bZ)) = J(H_1(X, \bZ)),\]
where, for any $\bZ$-mixed Hodge structure $H$ whose weights are non-positive, $J(H) := H_{\bC} \slash (F^0 + H_{\bZ})$ is the corresponding intermediate Jacobian \cite{Carlson}. Therefore, the map $a$ is the period map associated to a $\bZ$-variation of mixed Hodge structure of geometric origin (hence admissible). In particular, it is holomorphic. If $\dim X \geq 2$, then this implies that the map extends as a holomorphic map $a \colon X \to J(H_1(X, \bZ))$. An easy computation shows that this is also the case if $\dim X = 1$.\\

The push-out along the morphism $H_1(X, \bZ) \to H$ yields a map 
\[\mathrm{Ext}^1_{\bZ-MHS}(\bZ(0) , H_1(X, \bZ)) \to \mathrm{Ext}^1_{\bZ-MHS}(\bZ(0) , H),\]
or equivalently a map $J(H_1(X, \bZ)) \to J(H)$ which is holomorphic.
The composite map $X \to  J(H)$ is the period map associated to an admissible $\bZ$-variation of mixed Hodge structure with monodromy representation $\pi_1(X) \to H$. Thanks to \cite[Lemma 2.3]{BBT-mixed-Griffiths}, the period map $X \to  J(H)$ is proper if and only if the pair $(X, \rho)$ is maximal. Moreover, by definition of $H$, for every connected complex algebraic variety $Z$ equipped with a morphism $Z \to X$, the induced homomorphism $\pi_1(Z) \to H$ is zero if and only if the composite map $Z \to X \to J(H)$ is constant. This completes the proof of Theorem \ref{Shafarevich in the abelian case}.
\end{proof}

\bibliographystyle{alpha}
\bibliography{biblio}

\begin{thebibliography}{BDDM22}

\bibitem[AC23]{Aguilar-Campana}
Rodolfo~Aguilar Aguilar and Frédéric Campana.
\newblock The nilpotent quotients of normal quasi-projective varieties with
  proper quasi-{A}lbanese map, 2023.

\bibitem[Bas80]{Bass}
Hyman Bass.
\newblock Groups of integral representation type.
\newblock {\em Pacific J. Math.}, 86(1):15--51, 1980.

\bibitem[BBT20]{BBT-mixed-Griffiths}
Benjamin Bakker, Yohan Brunebarbe, and Jacob Tsimerman.
\newblock Quasiprojectivity of images of mixed period maps, 2020.

\bibitem[BDDM22]{Brotbek-Daskalopoulos-Deng-Mese}
Damian Brotbek, Georgios Daskalopoulos, Ya~Deng, and Chikako Mese.
\newblock Representations of fundamental groups and logarithmic symmetric
  differential forms, 2022.

\bibitem[Ber90]{Berkovich}
Vladimir~G. Berkovich.
\newblock {\em Spectral theory and analytic geometry over non-{A}rchimedean
  fields}, volume~33 of {\em Mathematical Surveys and Monographs}.
\newblock American Mathematical Society, Providence, RI, 1990.

\bibitem[Cam94]{Campana94}
Fr{\'e}d{\'e}ric Campana.
\newblock Remarques sur le rev\^etement universel des vari\'et\'es
  k\"ahl\'eriennes compactes.
\newblock {\em Bull. Soc. Math. France}, 122(2):255--284, 1994.

\bibitem[Car60]{Cartan}
Henri Cartan.
\newblock Quotients of complex analytic spaces.
\newblock In {\em Contributions to function theory ({I}nternat. {C}olloq.
  {F}unction {T}heory, {B}ombay, 1960)}, pages 1--15. Tata Institute of
  Fundamental Research, Bombay, 1960.

\bibitem[Car80]{Carlson}
James~A. Carlson.
\newblock Extensions of mixed {H}odge structures.
\newblock In {\em Journ\'{e}es de {G}\'{e}ometrie {A}lg\'{e}brique d'{A}ngers,
  {J}uillet 1979/{A}lgebraic {G}eometry, {A}ngers, 1979}, pages 107--127.
  Sijthoff \& Noordhoff, Alphen aan den Rijn---Germantown, Md., 1980.

\bibitem[CCE15]{CCE15}
Fr\'{e}deric Campana, Beno\^{\i}t Claudon, and Philippe Eyssidieux.
\newblock Repr\'{e}sentations lin\'{e}aires des groupes k\"{a}hl\'{e}riens:
  factorisations et conjecture de {S}hafarevich lin\'{e}aire.
\newblock {\em Compos. Math.}, 151(2):351--376, 2015.

\bibitem[CDY22]{Cadorel-Deng-Yamanoi}
Benoit Cadorel, Ya~Deng, and Katsutoshi Yamanoi.
\newblock Hyperbolicity and fundamental groups of complex quasi-projective
  varieties, 2022.

\bibitem[Cor88]{Corlette}
Kevin Corlette.
\newblock Flat {$G$}-bundles with canonical metrics.
\newblock {\em J. Differential Geom.}, 28(3):361--382, 1988.

\bibitem[EKPR12]{EKPR}
P.~Eyssidieux, L.~Katzarkov, T.~Pantev, and M.~Ramachandran.
\newblock Linear {S}hafarevich conjecture.
\newblock {\em Ann. of Math. (2)}, 176(3):1545--1581, 2012.

\bibitem[Eys04]{Eyssidieux-reductive}
Philippe Eyssidieux.
\newblock Sur la convexit\'{e} holomorphe des rev\^{e}tements lin\'{e}aires
  r\'{e}ductifs d'une vari\'{e}t\'{e} projective alg\'{e}brique complexe.
\newblock {\em Invent. Math.}, 156(3):503--564, 2004.

\bibitem[Fuj15]{Fujino-Albanese}
O.~Fujino.
\newblock On quasi-{A}lbanese maps.
\newblock {\em preprint}, 2015.

\bibitem[GGK22]{Green-Griffiths-Katzarkov}
Mark Green, Phillip Griffiths, and Ludmil Katzarkov.
\newblock Shafarevich mappings and period mappings, 2022.

\bibitem[Gri70]{GriffithsIII}
Phillip~A. Griffiths.
\newblock Periods of integrals on algebraic manifolds. {III}. {S}ome global
  differential-geometric properties of the period mapping.
\newblock {\em Inst. Hautes \'{E}tudes Sci. Publ. Math.}, (38):125--180, 1970.

\bibitem[GS92]{Gromov-Schoen}
Mikhail Gromov and Richard Schoen.
\newblock Harmonic maps into singular spaces and {$p$}-adic superrigidity for
  lattices in groups of rank one.
\newblock {\em Inst. Hautes \'Etudes Sci. Publ. Math.}, (76):165--246, 1992.

\bibitem[Kat97]{Katzarkov97}
L.~Katzarkov.
\newblock Nilpotent groups and universal coverings of smooth projective
  varieties.
\newblock {\em J. Differential Geom.}, 45(2):336--348, 1997.

\bibitem[Kol93]{Kollar-Shafarevich}
J\'{a}nos Koll\'{a}r.
\newblock Shafarevich maps and plurigenera of algebraic varieties.
\newblock {\em Invent. Math.}, 113(1):177--215, 1993.

\bibitem[KR98]{Katzarkov-Ramachandran}
L.~Katzarkov and M.~Ramachandran.
\newblock On the universal coverings of algebraic surfaces.
\newblock {\em Ann. Sci. \'{E}cole Norm. Sup. (4)}, 31(4):525--535, 1998.

\bibitem[LM85]{Lubotzky-Magid}
Alexander Lubotzky and Andy~R. Magid.
\newblock Varieties of representations of finitely generated groups.
\newblock {\em Mem. Amer. Math. Soc.}, 58(336):xi+117, 1985.

\bibitem[LR96]{Lasell-Ramachandran}
Brendon Lasell and Mohan Ramachandran.
\newblock Observations on harmonic maps and singular varieties.
\newblock {\em Ann. Sci. \'{E}cole Norm. Sup. (4)}, 29(2):135--148, 1996.

\bibitem[MFK94]{Mumford-GIT}
D.~Mumford, J.~Fogarty, and F.~Kirwan.
\newblock {\em Geometric invariant theory}, volume~34 of {\em Ergebnisse der
  Mathematik und ihrer Grenzgebiete (2) [Results in Mathematics and Related
  Areas (2)]}.
\newblock Springer-Verlag, Berlin, third edition, 1994.

\bibitem[Moc06]{Mochizuki_asterisque}
Takuro Mochizuki.
\newblock Kobayashi-{H}itchin correspondence for tame harmonic bundles and an
  application.
\newblock {\em Ast\'{e}risque}, (309):viii+117, 2006.

\bibitem[Moc07]{Mochizuki}
Takuro Mochizuki.
\newblock Asymptotic behaviour of tame harmonic bundles and an application to
  pure twistor {$D$}-modules. {I, II}.
\newblock {\em Mem. Amer. Math. Soc.}, 185(869, 870):xii+324, 2007.

\bibitem[Pro74]{Procesi}
Claudio Procesi.
\newblock Finite dimensional representations of algebras.
\newblock {\em Israel J. Math.}, 19:169--182, 1974.

\bibitem[Sha13]{Shafarevich_book}
Igor~R. Shafarevich.
\newblock {\em Basic algebraic geometry. 2}.
\newblock Springer, Heidelberg, third edition, 2013.
\newblock Schemes and complex manifolds, Translated from the 2007 third Russian
  edition by Miles Reid.

\bibitem[Sik12]{Sikora}
Adam~S. Sikora.
\newblock Character varieties.
\newblock {\em Trans. Amer. Math. Soc.}, 364(10):5173--5208, 2012.

\bibitem[Sim92]{Simpson_Higgs}
Carlos~T. Simpson.
\newblock Higgs bundles and local systems.
\newblock {\em Inst. Hautes \'Etudes Sci. Publ. Math.}, (75):5--95, 1992.

\bibitem[Sim94]{Simpson-ModuliII}
Carlos~T. Simpson.
\newblock Moduli of representations of the fundamental group of a smooth
  projective variety. {II}.
\newblock {\em Inst. Hautes \'{E}tudes Sci. Publ. Math.}, (80):5--79 (1995),
  1994.

\end{thebibliography}

\end{document}